\documentclass[11pt,reqno]{amsart}
\usepackage{amssymb}
\usepackage{txfonts}
\usepackage{amsfonts}
\usepackage{amsmath}
\usepackage{graphicx}
\usepackage{hyperref}
\usepackage{float}
\usepackage{epstopdf}
\usepackage{color}

\allowdisplaybreaks
\newtheorem{theorem}{Theorem}[section]
\newtheorem{proposition}[theorem]{Proposition}
\newtheorem{lemma}[theorem]{Lemma}
\newtheorem{corollary}[theorem]{Corollary}

\newtheorem{definition}[theorem]{Definition}

\def\diam{\mathrm{diam}}
\def\CH{\mathcal{CH}}

\begin{document}

\title{Metrics on the Sierpinski carpet by weight functions}
\author{Qingsong Gu}
\address{Department of Mathematics\\ The Chinese University of Hong Kong, Hong Kong, China}
\email{001gqs@163.com}

\author{Hua Qiu}
\address{Department of Mathematics, Nanjing Univeristy, Nanjing 210093, China}
\email{huaqiu@nju.edu.cn}
\thanks{The research of Qiu was supported by the NSFC grant 11471157}

\author{Huo-Jun Ruan}
\address{School of Mathematical Science, Zhejiang University, Hangzhou 310027,
China} \email{ruanhj@zju.edu.cn}
\thanks{The research of Ruan was supported in part by
the NSFC grants 11271327, 11771391, and by ZJNSFC grant LR14A010001.}



\date{}



\begin{abstract}
We construct certain metrics on the Sierpinski carpet via a class of self-similar weight functions. Using these metrics and by applying known results, we obtain the two-sided sub-Gaussian heat kernel estimates of time change of the standard diffusion on the Sierpinski carpet with respect to self-similar measures. This proves a conjecture by Kigami.
\end{abstract}
\maketitle

\section{Introduction}\label{intro}
\setcounter{equation}{0}
Barlow and Bass created a successful theory about diffusions on the \textit{Sierpinski carpet} (denoted by $K$) in a  series of papers \cite{BB89,BB90-1,BB90-2,BB92,BB93,BB99}. They constructed a natural diffusion (also called the Brownian motion) on $K$, which is invariant under all the local isometries of $K$, and proved that this process has a transition probability, namely the \textit{heat kernel} $p_t(x,y)$, with respect to the normalized Hausdorff measure $\nu$ on $K$. Moreover, they obtained that $p_t(x,y)$ satisfies the two-sided sub-Gaussian estimates, i.e.,
\begin{equation*}
p_t(x,y)\asymp \frac{1}{t^{\alpha/\beta}}\exp\left(-c\left(\frac{|x-y|}{t^{1/\beta}}\right)^{\frac{\beta}{\beta-1}}\right),
\end{equation*}
where $\alpha=\frac{\log8}{\log3}$ is the Hausdorff dimension of $K$, and $\beta$ is called the walk dimension.
Subsequently, Kusuoka and Zhou\cite{KZ} gave a different method to  construct a diffusion on $K$, satisfying the self-similar identity, which has the same invariant properties as that of Barlow and Bass. Later, the uniqueness result of Barlow, Bass, Kumagai and Teplyaev \cite{BBKT} ensures that the different constructions yield a unique diffusion on $K$ up to scalar constants. Let $(\mathcal E,\mathcal F)$ be the associated local regular \textit{Dirichlet form} on $L^2(K,\nu)$. Let $\{F_i\}_{i=1}^8$ be the iterated function system associated with $K$. Then $(\mathcal E,\mathcal F)$ satisfies the following self-similar identity:
there is a constant $\rho>0$, for any $1\leq i\leq 8$, for any $u\in\mathcal F$, $u\circ F_i\in \mathcal F$, and
\begin{equation*}
\mathcal E(u)=\frac1\rho\sum_{i=1}^8\mathcal E(u\circ F_i),
\end{equation*}
where $\rho$ is called the \textit{renormalization factor} of $(\mathcal E,\mathcal F)$. It is known \cite{BB90-2} that $\rho$ is between $1.25147$ and $1.25149$.
Barlow and Bass also showed that the above mentioned results work for a class of infinitely ramified fractals, called the \textit{generalized Sierpinski carpets} (GSC) which includes the Sierpinski carpet as a special case. Please see \cite{BB99} for the definition of GSC.


There are many extensions of this theory. Barlow and Kumagai studied the time change of the process via self-similar measures in \cite{BK}. For a self-similar measure $\mu$ on a generalized Sierpinski carpet, they showed that the time change is possible if and only if $\mu_i\rho<1$ for all $i$'s, where $\mu_i$ is the $i$-th weight of $\mu$. In a series of papers \cite{Kig04,Kig09,Kig12,Kig18}, Kigami studied under what conditions can one have nice heat kernel bounds (e.g. two-sided sub-Gaussian) of a time change process on certain self-similar fractals (including \textit{p.c.f. self-similar sets} \cite{Kig01} and the generalized Sierpinski carpets). A necessary condition is requiring the measure $\mu$ to be volume doubling with respect to the resistance metric, or equivalently  the Euclidean metric, since the two metrics are quasisymmetrically equivalent, and in \cite{Kig09} there are criteria given by Kigami for the volume doubling property of a self-similar measure $\mu$.

In this paper, we restrict to consider the standard Sierpinski carpet $K$ equipped with a symmetric self-similar measure $\mu$. To be precise, let $F_i(x)=(x+2 p_i)/3, i=1,\ldots,8$, be contractive maps on $\mathbb R^2$, with $p_1=(0,0)$, $p_2=(1/2,0)$, $p_3=(1,0)$, $p_4=(1,1/2)$, $p_5=(1,1)$, $p_6=(1/2,1)$, $p_7=(0,1)$ and $p_8=(0,1/2)$. Then the Sierpinski carpet $K$ is the unique nonempty compact subset in $\mathbb{R}^2$ satisfying
\begin{equation*}
K=\bigcup\limits_{i=1}^8 F_i(K),
\end{equation*}
see Figure~\ref{figure: SC-fig}. Let $\mu$ be a self-similar measure supported on $K$, with the weights $\{\mu_i\}_{i=1}^8$ satisfying $\mu_1=\mu_3=\mu_5=\mu_7$, $\mu_2=\mu_4=\mu_6=\mu_8$ and $\mu_1+\mu_2=1/4$. To obtain a two-sided sub-Gaussian heat kernel estimates on $K$, the key step we need to do is to show the existence of a ``good" metric matching the given self-similar measure $\mu$. We apply the method in \cite{Kig09} to obtain a \textit{pseudometric} $D_g$ on $K$ from a given self-similar weight function $g$ which is defined on cells of $K$. Then the fundamental problem is to determine whether $D_g$ is indeed a metric on $K$.

\begin{figure}[htbp]
\begin{center}
  \scalebox{0.5}{\includegraphics{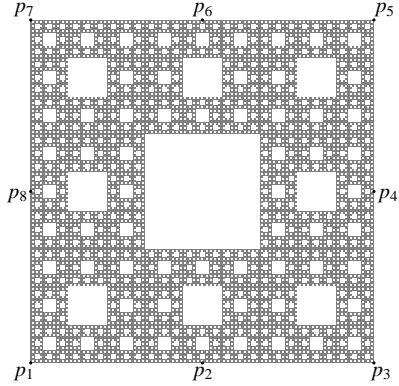}}
\end{center}
\caption{The Sierpinski carpet.}
\label{figure: SC-fig}
\end{figure}

More precisely, Let $\Sigma^0=\{\emptyset\}$. For $n\geq1$, let $\Sigma^n=\{1,\ldots,8\}^n$ be the collection of \textit{words} with length $n$. For $w=w_1\cdots w_n\in\Sigma^n$, we write $K_w=F_w(K):=F_{w_1}\circ\cdots\circ F_{w_n}(K)$, and call it an \textit{$n$-cell} of $K$. Denote by $\Sigma^*=\bigcup_{n\geq 0} \Sigma^n$ the collection of all finite words, and by $|w|$ the \textit{length} of $w$ for each $w\in \Sigma^*$.

Following \cite[Chapter 2]{Kig09}, a finite sequence of words $\big(w(1), \ldots,w(m)\big)$ in $\Sigma^*$, or equivalently, cells $\big(K_{w(1)}, \ldots,K_{w(m)}\big)$ in $K$ is called a \textit{chain} if $K_{w(i)}\cap K_{w(i+1)}\not=\emptyset$ for  $1\leq i\leq m-1$. A chain $\big(w(1), \ldots,w(m)\big)$ is said to be a chain  between $x$ and $y$ for $x,y\in K$ if $x\in K_{w(1)}$ and $y\in K_{w(m)}$.

We call $g: \Sigma^*\rightarrow [0,1]$ a \textit{weight function} if $g$ satisfies the following two conditions:

$1$.  $g(\emptyset)=1$, $g(w j)\leq g(w)$ if $w\in \Sigma^*$ and $j\in\{1,\ldots,8\}$;

$2$. $\lim\limits_{n\rightarrow\infty}\sup_{w \in {\Sigma^n}}g(w )=0$.

For a weight function $g$ and for any $x,y\in K$, we define
\begin{equation*}
  D_g(x,y)=\inf\Big\{ \sum_{i=1}^m g\big(w(i)\big) |\, (w(1),\ldots,w(m)\big) \mbox{ is a chain between $x$ and $y$}\Big\}.
\end{equation*}
It is easy to see that $D_g(\cdot,\cdot)$ is symmetric, nonnegative and satisfies the triangle inequality. However, in general, it may happen that $D_g(x,y)=0$ for some pairs $x\not=y$ in $K$ so that $D_g$  fail to be a metric. It is also clear that $D_g$ is continuous with respect to the Euclidean metric on $K$.

We focus on a class of self-similar weight functions. Given $(a,b)\in (0,1)^2$, we define $g_{a,b}: \Sigma^*\rightarrow (0,1]$ by $g_{a,b}(w)=\prod_{k=1}^n r_{w_k}$ for $w=w_1\cdots w_n$, where
\begin{equation}\label{eqeq1}
  r_{j}=\begin{cases}
     a, & \mbox{ if } j\in \{1,3,5,7\},\\
     b, & \mbox{ if } j\in \{2,4,6,8\},
  \end{cases}
\end{equation}
see Figure~\ref{figure:weight} for the cases $|w|=1$ and $|w|=3$.
It is easy to see that $g_{a,b}$ is a weight function.

\begin{figure}[htbp]
    \begin{minipage}[t]{0.45\linewidth}
      \begin{center}
        {\includegraphics[height=4cm]{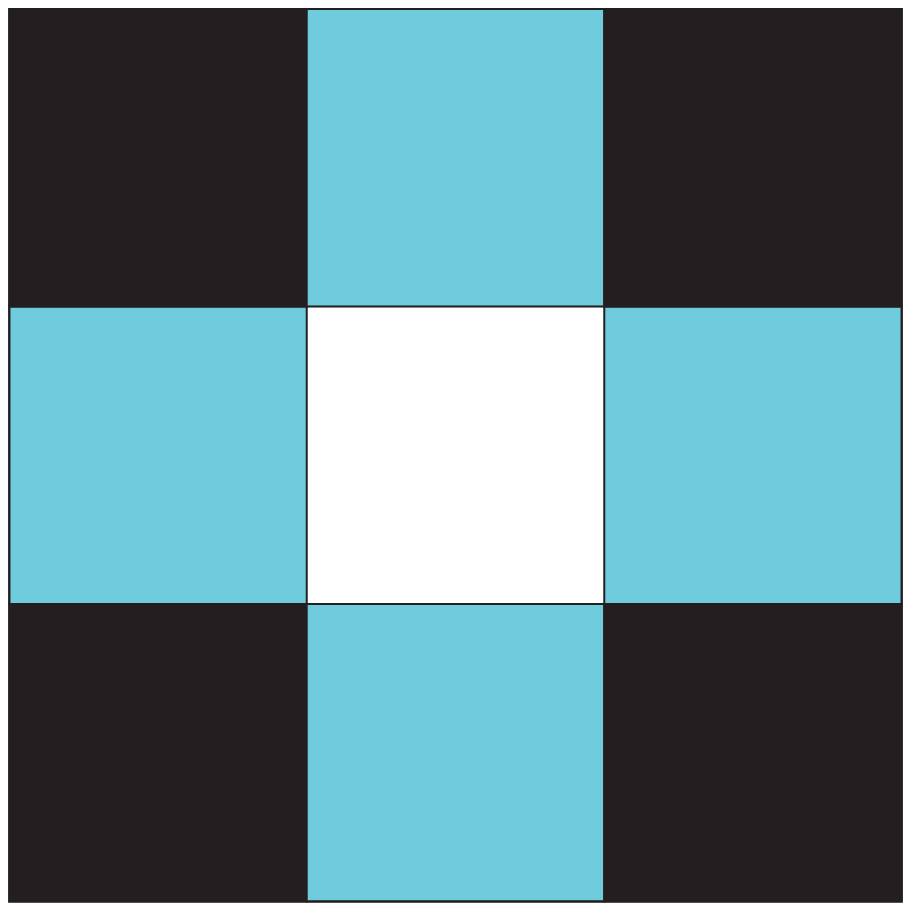}}
      \end{center}
  \end{minipage}
    \hfill
    \begin{minipage}[t]{0.45\linewidth}
       \begin{center}
           {\includegraphics[height=4cm]{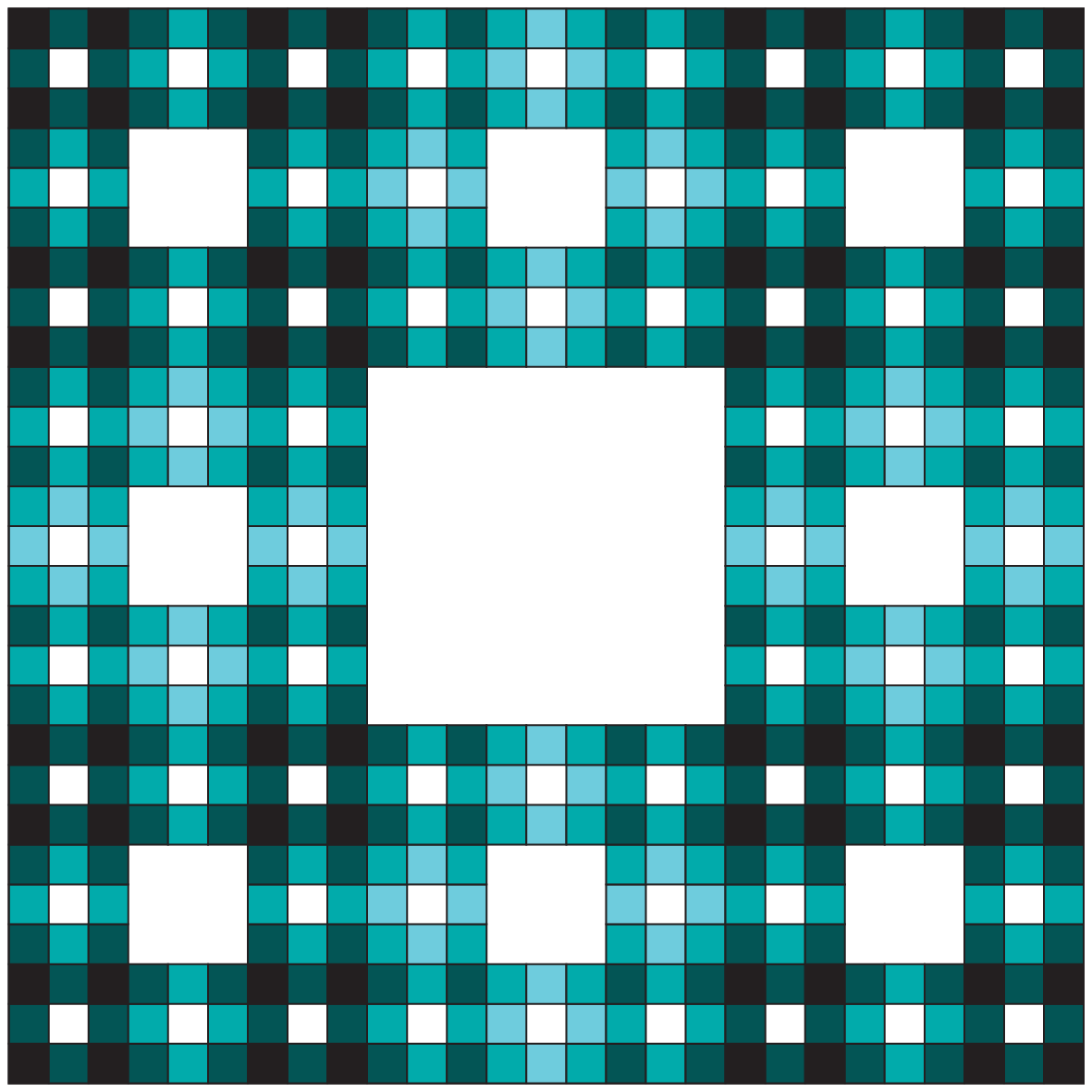}}
       \end{center}
    \end{minipage}
    \caption{The weight function $g_{a,b}$.}
    \label{figure:weight}
  \end{figure}

The main purpose of this paper is to characterize the equivalent conditions for $(a,b)\in (0,1)^2$ such that $D_{g_{a,b}}$ is a metric on $K$. Our result is the following (Theorem~\ref{th2.2}).

\begin{theorem}\label{th1.1}
  For $(a,b)\in (0,1)^2$, $D_{g_{a,b}}$ is a metric if and only if $2a+b\geq 1$ and $a+2b\geq 1$.
\end{theorem}
This result was conjectured by Kigami in a conference at Cornell University in 2017. He showed the ``only if\," part and claimed that the ``if\," part holds if additionally $b\geq 1/3$. For completeness, we will give a whole proof of the theorem. The following is our main idea to prove the ``if\," part: given a chain, we construct a new chain which has nice properties, while its total weight is comparable to that of the previous chain. We remark that the main technique is applying a series of operations on chains, see Section \ref{sec2} for details.

\medskip

Once we obtain the metrics on $K$, we want to study whether they satisfy the needs for the heat kernel estimates. A basic property is the {\it adaptedness} introduced by Kigami\cite{Kig09}, see Section \ref{sec3.1}. Another property is the {\it chain condition}, see \cite{GHL}, which is used in obtaining the off-diagonal lower bound of heat kernel.

\medskip

Denote $\sigma=\{(a,b)\in (0,1)^2:\, 2a+b\geq1 \textrm{ and } a+2b\geq1\}$. Let $\sigma_1$ be the set of elements $(a,b)$ in $\sigma$ such that $b\geq a$, and $\sigma_2$ be that with $a\geq b$, then $\sigma=\sigma_1\cup\sigma_2$. Write $I_1=\{(a,b)\in\sigma_1: 2a+b=1\}$ and $I_2=\{(a,b)\in \sigma_2: a+2b=1\}$. We give two properties of the constructed metrics (Theorems~\ref{thm-adapt}, Proposition \ref{thm-qs}):
\begin{theorem}\label{th1.2}
For $(a,b)\in \sigma$, $D_{g_{a,b}}$ is adapted to $g_{a,b}$.
\end{theorem}

\begin{proposition}\label{th1.3}
For $(a,b)\in \sigma$, $D_g$ is quasisymmetric to the Euclidean metric on $K$.
\end{proposition}

By using the adaptness of the metric, we then show that
\begin{theorem}\label{th1.4}
For $(a,b)\in\sigma$, $D_{g_{a,b}}$ satisfies the chain condition if and only if $(a,b)\in I_1\cup I_2$.
\end{theorem}
We remark that once we have shown that $D_g$ is a metric for some $g$ which is a self-similar weight function, the adapatedness of metric is an immediate consequence of \cite[Theorem 2.3.16]{Kig09} in a more general framework. We include a proof in this paper for completeness and also for the convenience of the readers.
After showing that the metric $d$ satisfies these conditions, and noticing that the fully symmetric self-similar measure $\mu$ satisfies the volume doubling property (see \cite{Kig09}), then by applying known results (e.g. \cite{Kig09}), we obtain the two-sided sub-Gaussian heat kernel estimates of the time change of $(\mathcal E,\mathcal F)$ associated with a given $\mu$.

\bigskip

We organize the paper as follows. In Section \ref{sec2}, we prove the main result Theorem \ref{th1.1}, which gives a positive answer to Kigami's conjecture. In Section \ref{sec3}, we study basic properties of the constructed metric, one is the adaptedness of the constructed metric to its weight function; the other is Theorem \ref{th1.3} for the criteria of the metric to satisfy the chain condition. By using these properties and applying known results established by Kigami, we obtain two-sided sub-Gaussian estimates for the heat kernel of the associated time change diffusions.


\section{Metrics on the Sierpinski carpet}\label{sec2}
 For a given chain $\gamma=\big(w(1),\cdots,w(m)\big)$, we define the \textit{total weight} of $\gamma$ by
\begin{equation*}
g_{a,b}(\gamma)=\sum_{i=1}^m g_{a,b}\big(w(i)\big).
 \end{equation*}
 The following property can be verified easily.
\begin{proposition}\label{th2.1}
For any two distinct words $w$ and $v$ with the same length such that $K_w$ and $K_v$ share a same line segment $L$, it holds that
\begin{equation*}
\frac{g_{a,b}(w)}{g_{a,b}(v)}=\frac ab \ \text{or}\ \frac ba.
\end{equation*}
Furthermore, if we reflect a chain $\gamma=(K_{w(1)},\ldots,K_{w(m)})$ contained in $K_v$ along $L$, and denote it by $R_L(\gamma)$, then $R_L(\gamma)$ is a chain contained in $K_w$ and $\frac{g_{a,b}(R_L(\gamma))}{g_{a,b}(\gamma)}=\frac ab \ \text{or}\ \frac ba.$

\end{proposition}

The main purpose of this section is to prove the following.
\begin{theorem}\label{th2.2}
  For $(a,b)\in (0,1)^2$, $D_{g_{a,b}}$ is a metric if and only if $2a+b\geq 1$ and $a+2b\geq 1$.
\end{theorem}

 We separate the proof of the theorem into three parts: the ``only if\," part, the ``if\," part for $(a,b)\in\sigma_1$ and the ``if\," part for $(a,b)\in\sigma_2$.
In what follows, we write $g$ instead of $g_{a,b}$, and $D_g$ instead of $D_{g_{a,b}}$ for simplicity.
\medskip

We remark that the proofs given below strongly rely on the fact that the straight line segments $\overline{p_1p_3}$ and $\overline{p_2p_4}$ are entirely contained in the Sierpinski carpet $K$.

\medskip

The proof of the ``only if\," part is straightforward.

\begin{proof}[Proof of the ``only if\," part of Theorem \ref{th2.2}]
Assume that $D_g$ is a metric on $K$ for some $(a,b)\in (0,1)^2$.

Firstly, for $n\geq1$, we denote $\gamma_n$ to be the chain between $p_1$ and $p_3$ given by a sequence of words with length $n$ such that the cells intersect the bottom line $\overline{p_1p_3}$. That is, $\gamma_1=\{1,2,3\}$, $\gamma_2=\{11,12,13,21,22,23,31,32,33\}$, and so on. By elementary calculations, we obtain
\begin{equation*}
g(\gamma_n)=(2a+b)^n.
\end{equation*}
Since $D_g$ is a metric, $\inf_{n\in \mathbb{Z}^+} (2a+b)^n \geq D_g(p_1,p_3)>0$ so that $2a+b\geq 1$.

Secondly, for $n\geq1$, we denote $\gamma'_n$ to be the chain between $p_2$ and $p_4$ given by a sequence of words with length $n$ such that the cells intersect the straight line $\overline{p_2p_4}$. That is, $\gamma_1^\prime=\{2,3,4\}$, $\gamma_2^\prime=\{22,23,24,38,37,36,42,43,44\}$, and so on. By elementary calculations, we obtain
\begin{equation*}
g(\gamma'_n)=(a+2b)^n.
\end{equation*}
Similarly as above, we have $a+2b\geq1$.
\end{proof}

The proof of the ``if\," part is much more awkward. We first introduce some notations and lemmas.

For any subset $E$ of $\mathbb{R}^2$, we denote by $\partial (E)$ the (Euclidean) boundary of $E$. For any point $p\in \mathbb{R}^2$, we denote by $x_p$ and $y_p$ the $x$-coordinate and $y$-coordinate of $p$, respectively. For any subset $E$ of $\mathbb{R}^2$, we denote by $\pi_x(E)${( similar for $\pi_y(E)$) the orthogonal projection of $E$ onto the $x$-axis, i.e. $\pi_x(E)=\{x_p:\, p\in E\}$.

Given a chain $\gamma=\big(w(1), \ldots,w(m)\big)$, we define the \emph{union} of cells in $\gamma$ by
\begin{equation*}
  \cup \gamma=\bigcup_{i=1}^m K_{w(i)},
\end{equation*}
and call $\big(w(i), \ldots,w(j)\big)$ a \emph{sub-chain} of $\gamma$ for any $i,j\in \{1,\ldots,m\}$ with $i<j$.

For any two subsets $X,Y$ of $K$ and any $z\in K$, we define
\begin{align*}
  &D_g(X,Y)=\inf\{D_g(x,y):\,x\in X,y\in Y\}, \quad \textrm{and} \\
  &D_g(z,X)=\inf\{D_g(z,x):\,x\in X\}.
\end{align*}

For any word $w=w_1w_2\cdots w_n\in \Sigma^*$ and $k\leq n$, we denote $w|_k=w_1w_2\cdots w_k$. Given $i\in \Sigma^1$, we call $wi=w_1 w_2\cdots w_n i$ a \emph{son} of $w$. Furthermore, $wi$ is called an \emph{$a$-son} of $w$ if $i\in \{1,3,5,7\}$, and a \emph{$b$-son} of $w$ if $i\in\{2,4,6,8\}$.

\begin{lemma}\label{lemma1}
  For  $(a,b)\in (0,1)^2$, let $g$ be the weight function as defined in (\ref{eqeq1}). If $D_g$ is not a metric on $K$, then $D_{g}(x,y)=0$ for all $x,y\in K$.
\end{lemma}
\begin{proof}
Let $q,s$ be two distinct points in $K$ with $D_g(q,s)=0$. We pick a positive integer $n$ satisfying $2\cdot 3^{-n}<\max\{|x_q-x_s|,|y_q-y_s|\}$. Without loss of generality, we assume that $x_s>x_q+2\cdot 3^{-n}$. Define $\alpha_n=\lceil  3^n x_q\rceil \cdot 3^{-n}$, where $\lceil t \rceil$ is the minimum integer greater than or equal to $t$.

Let $m_0\in \mathbb{Z}^+$ satisfy $m_0^{-1}<\min \{ a^n, b^n\}$.
From $D_g(q,s)=0$, we know that for each positive integer $m\geq m_0$, there exists a chain $\gamma_m$ between $q$ and $s$ such that $g(\gamma_m)<m^{-1}$. From the definition of $m_0$, every cell in $\gamma_m$ has length greater than $n$. Thus, we may pick two points $q_m$ and $s_m$ in $K$, and a sub-chain $\widetilde{\gamma}_m$ of $\gamma_m$ between $q_m$ and $s_m$, such that $x_{q_m}=\alpha_n$, $x_{s_m}=\alpha_n+3^{-n}$, and $\pi_x(\cup \widetilde{\gamma}_m)=[\alpha_n,\alpha_n+3^{-n}]$. By taking subsequence if necessary, we may assume that there exist two $n$-cells $K_w$ and $K_v$ such that $q_m\in K_w$ and $s_m\in K_v$ for all $m\geq m_0$, and $\pi_x(K_w)=\pi_x(K_v)=[\alpha_n,\alpha_n+3^{-n}]$.

Denote $\beta_n=\min\big(\pi_y(K_w)\big)$.
For each $m\geq m_0$, we define
\begin{align*}
  Y_m=\lceil 3^n \max\big(\pi_y(\cup \widetilde{\gamma}_m)\big)\rceil \cdot 3^{-n}, \qquad y_m=\lfloor 3^n \min\big(\pi_y(\cup \widetilde{\gamma}_m)\big) \rfloor\cdot 3^{-n},
\end{align*}
where $\lfloor t \rfloor$ is the maximum integer less than or equal to $t$. If $Y_m>\beta_n+3^{-n}$, for all cells of $\widetilde{\gamma}_m$ in $[\alpha_n,\alpha_n+3^{-n}]\times [Y_m-3^{-n},Y_m]$, we reflect them along the line $y=Y_m-3^{-n}$. Similarly, if $y_m<\beta_n$, we reflect all cells of $\widetilde{\gamma}_m$ in $[\alpha_n,\alpha_n+3^{-n}]\times [y_m,y_m+3^{-n}]$ along the line $y=y_m+3^{-n}$. We do this repeatedly until all cells are contained in $K_w$. Then we obtain a point $s_m^\prime\in K_w$ with $x_{s_m^\prime}=\alpha_n+3^{-n}$, and a chain $\gamma_m^\prime$ between $q_m$ and $s_m^\prime$ such that $\cup \gamma_m^\prime \subset K_w$. By Proposition~\ref{th2.1}, we know that
\begin{equation}\label{eq:c-def}
  g(\gamma_m^\prime)\leq c^{3^n} g(\gamma_m)<c^{3^n} m^{-1}, \quad \textrm{ where } c=\max\{a/b,b/a\}.
\end{equation}



By taking subsequence if necessary, we may assume that $q_m$ converges to $q^*$, and $s_m^\prime$ converges to $s^*$.
Then $q^*,s^*\in K_w$ with $x_{q^*}=\alpha_n$, $x_{s^*}=\alpha_n+3^{-n}$, and $D_g(q^*,s^*)=0$.






Using the self-similarity, we can dilate $q^{*}, s^*$ to the two opposite sides of $K$. From above, we may assume that $q\in \overline{p_1p_7}$, $s\in \overline{p_3p_5}$. Let $q'=(1-x_q,y_q)$ and $s'=(1-x_s,y_s)$.
 Then $q^\prime\in \overline{p_3p_5}$ and $s^\prime\in \overline{p_1p_7}$. By symmetry, $D_{g}(s', q')=0$. Using self-similarity, we have $D_{g}\big(F_i(q),F_i(s)\big)=D_{g}\big(F_i(s'), F_i(q')\big)=0$ for $i=1,\cdots,8$.

Notice that $F_7(s)=F_6(s')$ and $F_5(q)=F_6(q')$. Thus
\begin{align*}
    &D_{g}\big(F_7(q), F_5(s)\big)\\
    &\leq D_{g}\big(F_7(q), F_7(s)\big) +
    D_{g}\big(F_7(s), F_5(q)\big) +
    D_{g}\big(F_5(q), F_5(s)\big) \\
    &=D_{g}\big(F_7(q), F_7(s)\big) +
    D_{g}\big(F_6(s'), F_6(q')\big) +
    D_{g}\big(F_5(q), F_5(s)\big) =0.
\end{align*}
 By using this repeatedly, we see that
  $D_{g}\big(F_{7^n}(q), F_{5^n}(s)\big)=0$ for all integers $n\geq 0$.
  By the continuity of $D_{g}$, we have $D_{g}(p_7,p_5)=0$.
  Then by symmetry, $D_{g}(p_1,p_7)=D_{g}(p_1,p_3)=D_{g}(p_3,p_5)=0$.

  Since any two points in $K$ can be connected by at most countably many vertical and horizontal line segments contained in $K$, by using the self-similarity and the triangle inequality of $D_g$, we must have that $D_{g}(x,y)=0$ for any $x,y\in K$. This completes the proof.
\end{proof}

\begin{lemma}\label{lemma4}
Let $(a,b)\in\sigma_2$, and $\gamma$ be a chain between $p_2$ and $p_8$. Assume that all the cells in $\gamma$ intersect $\overline{p_2p_8}$. Then
\begin{equation*}
g(\gamma)\geq\frac b{2a+b}.
\end{equation*}
\end{lemma}
\begin{proof}
First, we notice that the following fact holds: given a word $w\in \Sigma^*$ with $K_w\cap \overline{p_2p_8}\not=\emptyset$, there exist exactly one $a$-son and two $b$-sons of $w$, such that they intersect $\overline{p_2p_8}$.

 Let $w$ be a finite word such that $K_w$ intersects $\overline{p_2p_8}$. We call a word $u$ of the same length with $w$ an \textit{``on-line neighbor"} of $w$ if $K_u$ intersects $\overline{p_2p_8}$ and $K_w\cap K_u$ is a line segment. From the above fact, it is clear that $w$ has one or two on-line neighbors, where ``one" only happens when $p_2$ or $p_8$ is in $K_w$.

For each $w$ in $\gamma$, we denote by $J(w)$ the union of $w$ and all its on-line neighbors. Then it is clear that
 \begin{equation}\label{eq1}
\sum\limits_{u\in J(w)}g(u)\leq (1+2a/b)g(w).
\end{equation}


 Let $\Lambda(\gamma)=\bigcup\limits_{w\in \gamma}J(w)$. We claim that
\begin{equation*}
 \overline{p_2p_8}\subset \bigcup\limits_{u\in\Lambda(\gamma)}K_u.
\end{equation*}
In fact, for $w\in\gamma$, we define $\pi(K_w)$ to be the orthogonal projection of the cell $K_w$ on the line segment $\overline{p_2p_8}$. Since $\gamma$ is connected and $\pi$ is continuous, we have
\begin{equation*}
  \overline{p_2p_8}\subset\bigcup\limits_{w\in \gamma}\pi(K_w).
\end{equation*}
Combining this with another fact that
   $\pi(K_w) \subset\bigcup\limits_{u\in J(w)}K_u,$
we know that the claim holds.

By (\ref{eq1}), we have
 \begin{equation}\label{eq2}
\sum\limits_{u\in\Lambda(\gamma)}g(u)\leq \sum\limits_{w\in \gamma}\sum\limits_{u\in J(w)}g(u)\leq(1+2a/b)g(\gamma).
\end{equation}
 If there are two words $u$ and $v$ in $\Lambda(\gamma)$ satisfying $K_u\subset K_v$, we remove $u$ from $\Lambda(\gamma)$. Do this repeatedly on $\Lambda(\gamma)$ until there are no two such words. Let $\Lambda_1(\gamma)$ be the final set, then $\Lambda_1(\gamma)$ has the following three properties:
\begin{enumerate}
  \item[1.] any two cells in $\Lambda_1(\gamma)$ are not contained in each other;
  \item[2.] $\overline{p_2p_8}\subset\Lambda_1(\gamma);$
  \item[3.] $\sum\limits_{u\in\Lambda_1(\gamma)}g(u)\leq \sum\limits_{u\in\Lambda(\gamma)}g(u).$
\end{enumerate}


Let $k=\max\{|w|:w\in\Lambda_1(\gamma)\}$. Pick one word $u$ in $\Lambda_1(\gamma)$ such that $|u|=k$. Let $u^*=u|_{k-1}$. Then $\Lambda_1(\gamma)$ must contains the $a$-son and two $b$-sons of $u^*$ intersecting $\overline{p_2p_8}$. Now replace in $\Lambda_1(\gamma)$ the three sons by $u^*$. By $a+2b\geq1$, we have
 \begin{equation}\label{eq3}
g(u^*)\leq (a+2b)g(u^*)=\text{the total weight of the three sons of $u^*$}.
\end{equation}
Replace in $\Lambda_1(\gamma)$ all the $k$-cells by $(k-1)$-cells in the pattern ``three to one". Then do this for $(k-1)$-cells and so on. Finally, we obtain a $0$-cell.
By (\ref{eq3}), we see that in each replacement, $\sum\limits_{u\in\Lambda_1(\gamma)}g(u)$ does not increase, thus we have
 \begin{equation}\label{eq4}
\sum\limits_{u\in\Lambda(\gamma)}g(u)\geq\sum\limits_{u\in\Lambda_1(\gamma)}g(u)\geq g(\emptyset)=1.
\end{equation}
By (\ref{eq2}) and (\ref{eq4}), we have that
 \begin{equation*}
g(\gamma)\geq(1+2a/b)^{-1}=\frac b{2a+b}.
\end{equation*}
\end{proof}

Using a similar ``three to one'' argument, noticing that $2a+b\geq 1$ when $(a,b)\in\sigma_1$, we have the following lemma.
\begin{lemma}\label{lemma24}
Let $(a,b)\in\sigma_1$, and $\gamma$ be a chain between $p_1$ and $p_3$, assume that all the cells in $\gamma$ intersect $L_0=\overline{p_1p_3}$. Then
\begin{equation*}
g(\gamma)\geq1.
\end{equation*}
\end{lemma}

\begin{lemma}\label{lem-smash}
Assume that $D_g$ is not a metric on $K$. Then there exist two positive integers $N_0$ and $N_1$, such that for any two cells $K_w$ and $K_u$ with $|w|=|u|=N_0$, there is a chain $\gamma_{w,u}$ in $K$ starting with $K_w$ and ending with $K_u$ such that the following two conditions hold:
\begin{enumerate}
\item[1.] $g(\gamma_{w,u})\leq1;$
\item[2.] $N_0\leq |v|\leq N_1$ for any $v\in\gamma_{w,u}$.
\end{enumerate}
\end{lemma}
\begin{proof}
Let $N_0$ be the smallest integer such that $\max\{a^{N_0},b^{N_0}\}\leq\frac14$. Arbitrarily pick two cells $K_w$ and $K_u$ with $|w|=|u|=N_0$. In the case that $K_w\cap K_u\neq\emptyset$, we define $\gamma_{w,u}=\{w,u\}$ so that
\begin{equation*}
g(\gamma_{w,u})=g(w)+g(u)\leq2\max\{a^{N_0},b^{N_0}\}\leq\frac12.
\end{equation*}
In the case that $K_w\cap K_u=\emptyset$, we have $D_g(K_w,K_u)=0$ by using Lemma~\ref{lemma1}. Thus there is a chain $\eta_{w,u}$ connecting the two sets $K_w$ and $K_u$ with $g(\eta_{w,u})\leq\frac12$. Let $\gamma_{w,u}$ be the chain constructed by adding $\eta_{w,u}$ in between $w$ and $u$. Then we have
\begin{equation*}
g(\gamma_{w,u})=g(w)+g(\eta_{w,u})+g(u)\leq\frac12+\frac12=1.
\end{equation*}

Set $N_{w,u}=\max\{|v|:v\in\gamma_{w,u}\}.$
Let $N_1$ be the maximum of $N_{w,u}$ among all the pairs $w,u$ in $\Sigma^{N_0}$. Then the lemma holds with $N_0$ and $N_1$.
\end{proof}

\bigskip

%

The main idea of the proof of the ``if\," parts is that, suppose  $D_g$ is not a metric, then for a given chain $\gamma$, we do a series of operations on $\gamma$ to get a new chain $\widetilde\gamma$ satisfying the following two properties:

$1.$ there is a constant $C>0$ independent of $\gamma$ such that $g(\widetilde\gamma)\leq C g(\gamma)$;

$2.$ all the cells in $\widetilde\gamma$ should intersect a certain curve $\ell$.

\noindent Then by computation, one obtains that the total weight of any chain satisfying the second property has a positive lower bound to get a contradiction, and this implies that $D_g$ is a metric.

\subsection{The case for $\sigma_2$.}\label{sec2.1}
We first deal with the case when $(a,b)\in \sigma_2$. Before proceeding, we give some notations.

For a cell $K_w$, we define the \textit{center} of $K_w$ to be $F_w(1/2,1/2)$.

Throughout this subsection, for all $n\geq 0$, we denote $q_n=F_{3[6]^n}(1/2,1/2)$ and $q_n^\prime=F_{4[2]^n}(1/2,1/2)$, where we use $[i]^n$ to denote the word $w=w_1 w_2\cdots w_n$ with $w_k=i$ for all $1\leq k\leq n$. For each $n\geq 0$, we define $\ell_n$ and $\ell_n'$ to be the straight lines passing through the points $q_n$ and  $q'_n$ with the same slope $-1$, separately. Clearly, whenever $\ell_n$(or $\ell_n'$) intersects the interior of a cell with word length at least $n+1$, then the center of the cell lies in $\ell_n$(or $\ell_n'$).





Let $\ell_*$ and $\ell_*^\prime$ be the lines passing through the points $q_0$ and $q_0^\prime$ with the same slope $1$. For $n\geq 0$,  we denote by $M_n$ the rectangle enclosed by lines $\ell_*$, $\ell_*^\prime$, $\ell_n$ and $\ell_n^\prime$.

Let $\Omega$ be the hexagon enclosed by lines $y=2/9$, $y=4/9$, $\ell_0$, $\ell_0^\prime$, $\ell_*$ and $\ell_*^\prime$. See Figure~\ref{figure: omega-fig}.


\begin{figure}[htbp]
  \begin{center}
      {\includegraphics[height=5cm]{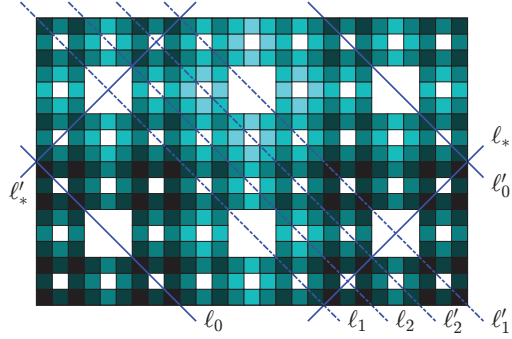}}
  \end{center}
  \caption{The hexagon $\Omega$ and lines.}
  \label{figure: omega-fig}
\end{figure}

It is easy to check that the following facts hold for all $n\geq 1$.
\begin{enumerate}
  \item Let $d_n=\frac{\sqrt{2}}{2}3^{-n-1}$. Then $d_n=d(\ell_{n},\ell_{n}^\prime)$, where $d(\ell_n,\ell_{n}^\prime)$ is the distance between the lines $\ell_n$ and $\ell_{n}^\prime$. Clearly, $d_n=d(\ell_{n-1},\ell_n)=d(\ell_n^\prime,\ell_{n-1}^\prime)$.
      Thus line $\ell_n^\prime$ is the reflection of $\ell_{n-1}$ through the line $\ell_n$. Similarly, line  $\ell_n$ is the reflection of $\ell_{n-1}^\prime$ through the line $\ell_n^\prime$.
  \item Let $K_w$ be an $(n+1)$-cell centered in $\Omega\cap M_{n-1}$. Then the center of $K_w$ must lie in $\ell_n,\ell_{n}^\prime$,$\ell_{n-1}$ or $\ell_{n-1}^\prime$. Thus, if $K_w$ is an $(n+1)$-cell centered in $\Omega\cap (M_{n-1}\setminus M_{n})$, then the center of $K_w$ must lie in $\ell_{n-1}$ or $\ell_{n-1}^\prime$. Furthermore, if we reflect an $(n+1)$-cell centered in $\Omega\cap \ell_{n-1}$ through $\ell_n$, then we obtain an $(n+1)$-cell centered in $\ell_n^\prime$. Similarly, if we reflect an $(n+1)$-cell centered in $\Omega\cap \ell_{n-1}^\prime$ through $\ell_n^\prime$, then we obtain an $(n+1)$-cell centered in $\ell_n$.
  \item Let $K_w$ be a cell centered in $\Omega\cap M_{n}$ with length at least $n+1$. Then $K_{w|_{n+1}}$ is also centered in $\Omega\cap M_{n}$.
  \item Let $K_w$ be a cell centered in $\Omega\cap (M_{n-1}\setminus M_n)$ with length at least $n+1$. If we reflect $K_w$ through $\ell_n$ if the center of $K_{w|_{n+1}}$ lies in $\ell_{n-1}\cup\ell_n$, or through $\ell_n^\prime$ otherwise, then we obtain a cell centered in $\Omega\cap M_n$. We also remark that the center of $K_{w|_{n+1}}$ lies in $\ell_{n-1}\cup\ell_n$ if and only if $K_w$ is centered in the closed strip between $\ell_{n-1}$ and $\ell_n$.
\end{enumerate}

\begin{lemma}\label{lem-reflection}
Let $n\geq 1$. For any cell $K_w$ centered in $\Omega\cap (M_{n-1}\setminus M_{n})$ with length at least $n+1$, let $K_u$ be the reflected cell of $K_w$ through $\ell_n$ if the center of $K_{w|_{n+1}}$ lies in $\ell_{n-1}\cup\ell_n$, or through $\ell_n^\prime$ otherwise. Then $K_u$ is centered in $\Omega\cap M_n$ and
\begin{equation}\label{eq:ref-ineq}
  g(u)\leq g(w).
\end{equation}
\end{lemma}
\begin{proof}
Assume that $w=w_1\cdots w_k$ and $u=u_1\cdots u_k$, where $k\geq n+1$.
Without loss of generality, we assume that the center of $K_{w|_{n+1}}$ lies in $\ell_{n-1}\cup\ell_n$. If the center of $K_{w|_{n+1}}$ lies on $\ell_n$, then $K_u$ is also a subcell of $K_{w|_{n+1}}$. By the symmetry, we have $r_{u_i}=r_{w_i}$ for all $i\geq n+2$ so that $g(u)= g(w).$

Now we assume that the center of $K_{w|_{n+1}}$ lies on $\ell_{n-1}$. In this case, the center of $K_{w|_n}$ also lies on $\ell_{n-1}$, and $w_{n+1}=3$ or $7$ so that $r_{w_{n+1}}=a$. It is easy to check that $K_{u|_{n}}$ and $K_{w|_{n}}$ share a same line segment so that $\frac{g(u|_n)}{g(w|_n)}$ is either $\frac{a}{b}$ or $\frac{b}{a}$.  Furthermore $u_{n+1}=2$ if $w_{n+1}=7$, and $u_{n+1}=8$ if $w_{n+1}=3$ so that $r_{u_{n+1}}=b$. Thus $\frac{g(u|_{n+1})}{g(w|_{n+1})}$ is either $\frac ba\cdot\frac ab$ or $\frac ba\cdot\frac ba$. By using that $b\leq a$, we have $g(u|_{n+1})\leq g(w|_{n+1}).$ By the symmetry, we have $r_{u_i}=r_{w_i}$ for all $i\geq n+2$ so that $g(u)\leq g(w).$
\end{proof}






Let $\ell$ be the line passing through the point $q=F_3(1/2,1)=F_4(1/2,0)$ with slope $-1$.
Then $\ell\cap \big(F_3(K)\cup F_4(K)\big)$ is a line segment contained in $K$. In fact, let $E=\overline{p_2p_4}\cup\overline{p_4p_6}\cup\overline{p_6p_8}\cup\overline{p_8p_2}$, then $E\subset \bigcup\limits_{i=1}^8 F_i(E)$ and hence $E\subset K$. Thus $\ell\cap \big(F_3(K)\cup F_4(K)\big)=F_3(\overline{p_4 p_6}) \cup F_4(\overline{p_8p_2})\subset K$.

We remark that if $K_w$ is an $(n+1)$-cell centered in $\Omega\cap M_n$, then $K_w$ is centered in $\ell_n$ or $\ell_n'$. Combining this with $\diam(K_w)= 2 d_n$, we know that $K_w$ intersects $\ell$.


\medskip

We introduce two chain operations as follows.

\textbf{$1.$ $n$-reflection.} Let $n\geq1$, one reflects each cell centered in $\Omega\cap (M_{n-1}\setminus M_n)$ with length at least $n+1$ through $\ell_n$ if the center of $K_{w|_{n+1}}$ lies in $\ell_{n-1}\cup\ell_n$, or through $\ell_n^\prime$ otherwise. We remark that from Lemma~\ref{lem-reflection}, given a chain in $\Omega$, with each cell centered in $\Omega\cap M_{n-1}$ and length at least $n+1$, then after doing $n$-reflection, we obtain a new chain in $\Omega$, with each cell centered in $\Omega\cap M_{n}$.

\textbf{$2.$ Smash.} Using the same method as in Lemma~\ref{lem-smash}, one replaces a cell $K_w$ by a finite sequence of its descendants with length at least $|w|+N_0$ and at most $|w|+N_1$ to get a chain. We remark that given a chain $\gamma$ and given some cells in $\gamma$,  we can do smash on these cells to obtain a new chain $\gamma^\prime$.




\medskip

We note that the operations ``smash" and ``reflection" do not increase the total weight of a chain in view of Lemma \ref{lem-smash} and Lemma \ref{lem-reflection}.


\begin{proof}[Proof of the ``if\," part for $\sigma_2$]
Let $(a,b)\in\sigma_2$, then $a+2b\geq1$ and $a\geq b$. Assume that $D_g$ is not a metric. Then by Lemma~\ref{lemma1} and using the similarity, there exist a sequence of chains $\gamma_1^{(n)}$ in $K_{36}$ between $F_{36}(0,1)$ and $F_{36}(1,0)$, and a sequence of chains $\gamma_2^{(n)}$ in $K_{42}$ between $F_{42}(0,0)$ and $F_{42}(0,1)$, such that $\lim_{n\to \infty} g(\gamma_i^{(n)})=0$ for $i=1,2$. From $F_{36}(0,1)=F_{42}(0,0)$, we know that $\gamma^{(n)}=\gamma_1^{(n)}\cup\gamma_2^{(n)}$ is a sequence of chains in $K_{36}\cup K_{42}$ between $F_{36}(1,0)$ and $F_{42}(0,1)$ satisfying $\lim_{n\to \infty} g(\gamma^{(n)})=0$.

Given a chain $\gamma_0$ contained in $K_{36}\cup K_{42}$, between $F_{36}(1,0)$ and $F_{42}(0,1)$, let
\begin{equation*}
k=\max\{|w|:w\in\gamma_0\},
\end{equation*}
and $n_0=\max\{k-N_1-1,0\}$, where $N_1$ is defined as in Lemma \ref{lem-smash}. From the assumption that $\cup\gamma_0\subset K_{36}\cup K_{42}$, we know that $\min\{|w|:w\in\gamma_0\}\geq 2$.

Now we do the following $n_0+1$ steps of operations on $\gamma_0$ to get a new chain $\widetilde{\gamma_0}$.

From step $1$ to step $n_0$, we do the following (if $n_0=0$, we simply skip this and do step $n_0+1$):

{\bf Step $i$ $(1\leq i\leq n_0)$.}\;
First, we do $i$-reflection on $\gamma_{i-1}$. Since each cell in $\gamma_{i-1}$ is centered in $\Omega\cap M_{i-1}$ with length at least $i+1$, we obtain a new chain $\gamma_{i-1}^\prime$, such that each cell in $\gamma_{i-1}^\prime$ is centered in $\Omega\cap M_{i}$. Then, we do smash on each word $w$ in $\gamma_{i-1}^\prime$ with length $i+1$, to obtain a new chain $\gamma_{i-1}^{\prime\prime}$. Notice that each cell in $\gamma_{i-1}^{\prime\prime}$ has length at least $i+2$, and the center of each cell lies in $\Omega\cap M_{i-1}$. Thus, doing $i$-reflection again, we obtain a new chain $\gamma_i$ with the following properties:

\emph{$g(\gamma_i)\leq g(\gamma_{i-1})$, each cell in $\gamma_i$ has length at least $i+2$ and at most $\max\{k,i+1+N_1\}$, and the center of each cell lies in $M_{i}$.}

We remark that in Step~$1$, we can directly do smash on $\gamma_0$ and then do $1$-reflection to obtain the new chain $\gamma_1$.

\medskip

From $n_0+1+N_1\geq k$, we know that each cell in $\gamma_{n_0}$ has length at least $n_0+2$ and at most $n_0+N_1+1$.

\medskip

At {\bf  step $n_0+1$}, we first do $(n_0+1)$-reflection to obtain a new chain $\gamma_{n_0}^\prime$. It is clear that each cell in $\gamma_{n_0}^\prime$ is centered in $\Omega\cap M_{n_0+1}$, with length at least $n_0+2$ and at most $n_0+N_1+1$. Now, we replace each cell $K_w$ in $\gamma_{n_0}^\prime$ by $K_{w|_{n_0+2}}$, to obtain a new chain $\gamma_{n_0+1}$. Then all cells in $\gamma_{n_0+1}$ are centered in $\Omega\cap M_{n_0+1}$ with length $n_0+2$  so that they intersect $\ell$. For each cell $K_w$ in $\gamma_{n_0}^\prime$, we have $g(w|_{n_0+2})\leq b^{-(N_1-1)}g(w)$ so that
\begin{equation*}
  g(\gamma_{n_0+1}) \leq b^{-(N_1-1)}g(\gamma_{n_0})\leq b^{-(N_1-1)}g(\gamma_0).
\end{equation*}

Let $\widetilde{\gamma}_0=\gamma_{n_0+1}$. Then $g(\widetilde{\gamma}_0)\leq b^{-(N_1-1)}g(\gamma_0)$ and each cell in $\widetilde{\gamma}_0$ intersects $\ell$.

Hence, let $n\geq 1$, for each chain $\gamma^{(n)}$, we obtain a new chain $\widetilde\gamma^{(n)}$ satisfying the two conditions below:

\begin{enumerate}
  \item[1.] each cell in $\widetilde\gamma^{(n)}$ intersects $\ell$;
  \item[2.]  $g(\widetilde\gamma^{(n)})\rightarrow0$ as $n\rightarrow\infty$.
\end{enumerate}

Finally, by the similarity and using  Lemma~\ref{lemma4}, we see that $g(\widetilde\gamma^{(n)})\geq (ab+b^2)\cdot\frac b{2a+b}$ for all $n\geq1$, a  contradiction, which implies that $D_g$ is a metric.
\end{proof}

\subsection{The case for $\sigma_1$.}\label{sec2.2}
Now we deal with the case $(a,b)\in \sigma_1$. In this case, $2a+b\geq1$ and $a\leq b$. We will use a similar trick as in the case $(a,b)\in  \sigma_2$.

For $n\geq 1$, we define $L_n$ and $L_n^\prime$ to be the line $y=\frac1{2\cdot3^{n-1}}$ and the line $y=\frac1{3^n}$, respectively. Then $L_n$ passes through $P_n$ and $L_n^\prime$ passes through $P_n^\prime$, where $P_n=F_{[1]^{n-1}}(0,1/2)$ and $P_n^\prime=F_{[1]^{n-1}}(0,1/3)$. Set $D_0^\prime$ to be the unit square $[0,1]^2$. For $n\geq 1$, we denote by $D_n$ (or $D_n^\prime$) the rectangle enclosed by $x$-axis, $y$-axis, and lines $x=1$ and $L_n$ (or $L_n^\prime$).

We state without proof the following lemma which is analogous to Lemma \ref{lem-reflection} in $\sigma_2$ case.
\begin{lemma}\label{lemma23}
Let $n\geq1$, for any cell $K_w$ centered in $D'_{n-1}\setminus D_{n}$ with length at least $n$, let $K_u$ be the reflected cell of $K_w$ along $L_n$, then
$K_u$ is centered in $D_n$ and
\begin{equation*}
g(u)=g(w);
\end{equation*}
 for any cell $K_w$ centered in $D_n\setminus D'_{n}$ with length at least $n$, let $K_u$ be the reflected cell of $K_w$ through $L'_n$, then $K_u$ is centered in $D'_{n}$ and
\begin{equation*}
g(u)=(a/b)\cdot g(w)\leq g(w).
\end{equation*}
\end{lemma}

\medskip

Then we will use chain operations to deal with the  $\sigma_1$ case as we did in $\sigma_2$ case. According to the new situation, we need the reflections in the following.

\textbf{$3.$ $L_n$-reflection.} Let $n\geq1$, one reflects each cell centered in $D'_{n-1}\setminus D_{n}$ with word length at least $n$ through $L_n$ from above to below.

\textbf{$4.$ $L'_n$-reflection.} Let $n\geq1$, one reflects each cell centered in $D_n\setminus D'_{n}$ with word length at least $n$ through $L_n'$ from above to below.

\medskip

In view of Lemma \ref{lemma23}, the operations $L_n$-reflection and $L_n^\prime$-reflection do not increase the total weight of a chain.
\begin{proof}[Proof of the ``if\," part for $\sigma_1$]
Let $(a,b)\in\sigma_1$. Then $2a+b\geq1$ and $a\leq b$. We will show that $D_g(p_1,p_3)\geq1$. Then in view of Lemma \ref{lemma1}, we will conclude that $D_g$ is a metric on $K$.

\bigskip

Given a chain $\gamma$ between $p_1$ and $p_3$. Denote $\gamma_0=\gamma$ and let
\begin{equation*}
k=\max\{|w|:w\in\gamma\}.
\end{equation*}
We do the following $k$ steps on $\gamma$ to get a new chain $\widetilde{\gamma}$.

{\bf Step $i$ ($1\leq i\leq k$).}\;
We do the $L_i$-reflection on $\gamma_{i-1}$ to get a new chain $\gamma_{i-1}^\prime$. Then we do the $L'_i$-reflection on $\gamma_{i-1}^\prime$ to obtain a new chain $\gamma_i$ with the following properties:

\emph{ $g(\gamma_i)\leq g(\gamma_{i-1})$, each cell in $\gamma_i$ with word length at most $i$ intersect the line $\overline{p_1p_3}$; and each cell with word length at least $i$ is centered in $D'_i$}.

\medskip

We denote $\gamma_k$ by $\widetilde{\gamma}$. From Lemma~\ref{lemma23}, we have
\begin{equation*}
g(\widetilde\gamma)\leq g(\gamma).
\end{equation*}
On the other hand, from Lemma~\ref{lemma24}, we have
\begin{equation*}
g(\widetilde\gamma)\geq1.
\end{equation*}
Hence we have $g(\gamma)\geq1$. Since this does not depend on the choice of $\gamma$, we conclude that
\begin{equation}\label{eq:Dgp1p3}
  D_g(p_1,p_3)\geq1.
\end{equation}
 Hence from Lemma~\ref{lemma1}, $D_g$ is a metric on $K$.
\end{proof}

\section{Application to sub-Gaussian heat kernel estimates}\label{sec3}
In this section, we are concerned with how we can use the metrics constructed to obtain the sub-Gaussian heat kernel estimates of the standard diffusion on the Sierpinski carpet \cite{BB92,KZ} by using time-changing via self-similar measures. We first study the properties of the metrics as constructed.

\subsection{adaptedness and quasisymmetry}\label{sec3.1}
As in the previous section, we still use $g$ and $D_g$ to represent $g_{a,b}$ and $D_{g_{a,b}}$ respectively.
Given a chain $\gamma=\big(w(1),\ldots,w(k)\big)$, we define $k$ to be the \textit{length} of the chain $\gamma$. Given $x,y\in K$, we denote by $\CH_k(x,y)$ to be the set of all chains between $x$ and $y$ with length $k$.

Let $M$ be a nonnegative integer. For any $p,q\in K$, we define
\begin{equation*}
  D_{g;M}(p,q)=\inf\left\{ g(\gamma)\Big|\; \gamma\in \CH_k(p,q) \textrm{ with } 1\leq k\leq M+1 \right\}.
\end{equation*}
From the definition, it is clear that $D_g(p,q)\leq D_{g;M}(p,q)$ for all $p,q\in K$. We say that $D_g$ is \emph{$M$-adapted} to $g$ if there exists a constant $C_*>0$ such that for all $p,q\in K$,
\begin{equation*}
   D_{g;M}(p,q) \leq C_* D_g(p,q).
\end{equation*}
$D_g$ is said to be \emph{adapted} to $g$ if $D_g$ is $M$-adapted to $g$ for some $M\geq 0$. It is clear that if $D_g$ is adapted to $g$, then $D_g$ is a metric. We remark that by Proposition~6.3 in \cite{Kig19}, the definitions of ``$D_g$ is $M$-adapted to $g$" and ``$D_g$ is adapted to $g$" in our setting are equivalent to the original definitions in \cite{Kig19}.

We first give a lemma to show that the metric $D_g$ constructed in Section \ref{sec2} has a lower bound for specific situations.
\begin{lemma}\label{lem-adapt}
  For $(a,b)\in \sigma$, there exists $C_0=C_0(a,b)>0$ such that
     \begin{equation*}
        \inf\{D_g(p,q):\, p,q\in K, \textrm{ and }|x_p-x_q|=1 \textrm{ or } |y_p-y_q|=1\}\geq C_0.
     \end{equation*}
\end{lemma}
\begin{proof}
By the symmetry, it suffices to show that for $(a,b)\in \sigma$,
\begin{equation*}
  \inf\{D_g(p,q):\, p\in \{0\}\times [0,1], q\in \{1\}\times [0,1]\}>0.
\end{equation*}
It is clear true since $D_g$ is continuous with respect to the Euclidean metric, and $\{0\}\times [0,1]$ and  $\{1\}\times [0,1]$ are two disjoint compact subsets of $K$.
\end{proof}

We remark that by definition, $C_0(a,b)\leq 1$ for all $(a,b)\in \sigma$.

In the following, for a given chain $\gamma$ in $K$, if all cells in $\gamma$ are contained in $K_w$ for some $w\in\Sigma^*$, then we say $\gamma$ is \textit{inside} $K_w$, and denote by
$F_w^{-1}(\gamma)$ the chain obtained by changing each cell $K_v$ in $\gamma$ into $F_w^{-1}(K_v)$.

\begin{theorem}\label{thm-adapt}
  $D_{g}$ is $1$-adapted to $g$ for all $(a,b)\in \sigma$.
\end{theorem}
\begin{proof}
Fix $p\not=q\in K$. Without loss of generality, we may assume that $|x_p-x_q|\geq |y_p-y_q|$. Let $m\in \mathbb{Z}^+$ satisfying $3^{-m}< |x_p-x_q|\leq 3^{-m+1}$. Then there exist two $(m+1)$-cells $u,u^\prime$, such that $p\in K_{u}$, $q\in K_{u^\prime}$ and $K_{u|_{m-1}}\cap K_{u^\prime|_{m-1}}\not=\emptyset$. Thus $$D_{g;1}(p,q)\leq g(u|_{m-1})+g(u^\prime|_{m-1})\leq (1+c^2)g(u|_{m-1})\leq (1+c^2)g(u)\max\{a^{-2},b^{-2}\},$$
where $c=\max\{a/b,b/a\}$.

Let $S$ be the square with center $F_u(1/2,1/2)$ and edge length equals $3^{-m}$. That is, $S$ is the union of the square $F_{u}([0,1]^2)$ and its eight neighboring squares with the same edge length. By $|x_p-x_q|>3^{-m}$, we know that $q$ does not lie in $S$.



Given a chain $\gamma=\big(w(1),\ldots,w(n)\big)$ between $p$ and $q$, we define
\begin{equation*}
  \ell=\min\{1\leq i\leq n:\, K_{w(i)} \textrm{ intersects the boundary of the square } S\}.
\end{equation*}

If $|w(\ell)|\geq  m+1$, then all cells in the chain $\gamma_1=\big(w(1),\ldots,w(\ell)\big)$ are contained in the square $S$. By reflecting all cells in $\gamma_1$ outside $K_u$ through the lines passing one of the four edges of the square $F_u([0,1]^2)$, we could always obtain a new chain $\gamma^*$ inside $K_u$ connecting two points $p^*$ and $q^*$, with $|x_{p^*}-x_{q^*}|=3^{-m-1}$, or $|y_{p^*}-y_{q^*}|=3^{-m-1}$. By the self-similarity and using Lemma~\ref{lem-adapt}, we have
$$g(\gamma^*)=g(u)g\big(F_u^{-1}(\gamma^*)\big)\geq g(u)C_0(a,b)$$
so that
\begin{equation}\label{eq:adapt}
  g(\gamma)\geq c^{-2} g(\gamma^*)\geq c^{-2}C_0(a,b)g(u).
\end{equation}

If $|w(\ell)|< m+1$, then there exists a subcell $v$ of $w(\ell)$ with $|v|=m+1$ such that either $v=u$ or $F_{v}(K)\cap F_u(K)$ is a point or a line segment. Thus,
\begin{equation*}
  g(\gamma)\geq g(v)\geq c^{-2}g(u)\geq c^{-2}C_0(a,b)g(u).
\end{equation*}

Let $C_*=c_1/c_2$, where $c_1=(1+c^2)\max\{a^{-2},b^{-2}\}$ and $c_2=c^{-2}C_0(a,b)$. Combining the above arguments, we have $D_{g;1}(p,q)\leq C_* g(\gamma)$ for every chain $\gamma$ between $p$ and $q$. It follows that $D_{g;1}(p,q)\leq C_* D_g(p,q)$. By the arbitrariness of $p$ and $q$, we know that the theorem holds.
\end{proof}

\bigskip

When $D_g$ is a metric on $K$, for a subset $E$ of $K$, we denote the diameter of $E$ under $D_g$ by
\begin{equation*}
\text{diam}_g(E)=\sup\{D_g(x,y):x,y\in E\}.
\end{equation*}

By using Theorem~6.4 in \cite{Kig19}, we can obtain the following result. We present a proof here for completeness.

\begin{corollary}\label{cor-adapt}
For $(a,b)\in \sigma$, $\diam_g(K_w)\asymp g(w)$ for all $w\in \Sigma^*$.
\end{corollary}
\begin{proof}
  Given $w\in \Sigma^*$, arbitrarily pick $p,q\in K_w$. Then $\gamma=(w)$ is chain between $p$ and $q$ with length $1$ so that $D_g(p,q)\leq g(w)$. This implies that $\diam_g(K_w)\leq g(w)$.

  On the other hand, we consider $p=F_w(0,0)$ and $q=F_w(1,1)$. For any chain $\gamma=(u,v)\in \CH_1(p,q)$, we can see that either $K_w\subset K_u$ or $K_w\subset K_v$ so that $g(\gamma) \geq g(w).$
  Thus $D_{g;1}(p,q)\geq g(w)$. Combining this with Theorem~\ref{thm-adapt}, we have
  \begin{equation*}
    \diam_g(K_w)\geq D_g(p,q)\geq C_*^{-1}D_{g;1}(p,q)\geq C_*^{-1} g(w),
  \end{equation*}
  where $C_*$ is the positive constant in the definition of $1$-adaptedness.
\end{proof}

Using the same argument, given a point $p\in K_w$ and a point $q\in K$ with $3^{-|w|-1}<\|p-q\|_\infty\leq 3^{-|w|}$, we have $D_g(p,q)\asymp g(w)$, where $\|p-q\|_\infty=\max\{|x_p-x_q|,|y_p-y_q|\}$.
By this, we can immediately obtain the quasisymmetric equivalence of the metrics $D_g$ and $d_\infty$, where $d_\infty$ is the metric on $K$ defined by $d_\infty(p,q)=\|p-q\|_\infty$. The following definition is from \cite{Kig12} for a metric space $M$ with two metrics $d$ and $\rho$. Some very interesting results on the quasisymmetric equivalence of Sierpinski carpets can be found in \cite{BM13, BM18pre} and the references therein.

\begin{definition}
$\rho$ is said to be quasisymmetric to $d$ if there exists a homeomorphism $h$ from $[0,\infty)$ to itself with $h(0)=0$ such that, for any $t>0$, $\rho(p,s)<h(t)\rho(p,q)$ whenever $d(p,s)<td(p,q)$.
\end{definition}
\begin{proposition}\label{thm-qs}
For $(a,b)\in \sigma$, $D_g$ is quasisymmetric to $d_\infty$ on $K$. As a result, $D_g$ is quasisymmetric to the classical Euclidean metric.
\end{proposition}
\begin{proof}
Clearly, $d_\infty$ is equivalent to the Euclidean metric. Thus it suffices to prove that there exist constants $C>0$ and $\kappa_1\geq\kappa_2>0$ such that for any $t>0$, $\|p-s\|_\infty<t\|p-q\|_\infty$ implies $D_g(p,s)<C\max\{t^{\kappa_1},t^{\kappa_2}\}D_g(p,q)$.

We first assume that $t\leq1$. Let $n,m\geq0$ be integers such that $3^{-n-1}<t\leq 3^{-n}$, and $3^{-m-1}<\|p-q\|_\infty\leq3^{-m}$. Let $w$ be a word such that $|w|=m$ and $p\in K_w$. Then $D_g(p,q)\asymp g(w)$. For any $s\in K$ such that $\|p-s\|_\infty<t\|p-q\|_\infty\leq 3^{-n-m}$, by using
$(a\vee b)^n \leq t^{-\log(a\vee b)/\log 3}\cdot (a\vee b)^{-1},$
 we have
\begin{align*}
  D_g(p,s)\leq C_1 g(w)\left(a\vee b\right)^n \leq C_2 D_g(p,q)t^{\frac{\log(a\vee b)}{-\log3}},
\end{align*}
where $a\vee b=\max\{a,b\}$, and $C_1,C_2$ are two positive constants independent of $p,q,s$.

Now we assume that $t>1$. Let $n,m\geq 0$ be two integers such that $3^{-n-1}<t^{-1}\leq 3^{-n}$, and $3^{-m-1}<\|p-s\|_\infty\leq3^{-m}$. Let $w$ be a word such that $|w|=m$ and $p\in K_w$. Then $D_g(p,s)\asymp g(w)$.
Let $q\in K$ be such that $\|p-s\|_\infty<t\|p-q\|_\infty$. Then $\|p-q\|_\infty>t^{-1}\|p-s\|_\infty>3^{-m-n-2}$ so that
\begin{align*}
D_g(p,q)\geq C_3 g(w)\left(a\wedge b\right)^n\geq C_4 D_g(p,s)t^{\frac{\log(a\wedge b)}{-\log3}},
\end{align*}
where $a\wedge b=\min\{a,b\}$, and $C_3,C_4$ are two positive constants independent of $p,q,s$.

Combining above two cases, our assertion follows by letting $C=\max\{C_2,C_4^{-1}\}$, $\kappa_1={\frac{\log(a\vee b)}{-\log3}}$ and $\kappa_2={\frac{\log(a\wedge b)}{-\log3}}$.
\end{proof}


\subsection{chain condition and heat kernel bounds}\label{sec3.2}
In order to obtain the two-sided sub-Gaussian heat kernel estimates, especially in obtaining the off-diagonal lower bound from near diagonal lower bounds by using a standard ``chain argument", one requires the metric to satisfy the \textit{chain condition}, see for example \cite{GHL}.
\begin{definition}
A metric space $(M,d)$ is said to satisfy the chain condition if there exists a constant $C>0$ such that for any two points $p,q\in M$ and for any positive integer $n$ there exists a sequence $\{q_i\}_{i=0}^n$ of points in $M$ such that $q_0=p$, $q_n=q$ and
\begin{equation}\label{eq:chain-cond}
d(q_i,q_{i+1})\leq C\frac{d(p,q)}n,\ \text{ for all }i=0,1,\cdots,n-1.
\end{equation}
\end{definition}
On the Sierpinski carpet, we know from above studies that there are many choices of $(a,b)$ to construct a metric $D_g$. However, the more interesting case is that $(a,b)$ are on the \textit{critical lines} $I_1=\{(a,b):\ 0<a\leq b<1,\ 2a+b=1\}$ and $I_2=\{(a,b):\ 0<b\leq a<1,\ a+2b=1\}$. We will show that a metric $D_g$ satisfies the chain condition if and only if $(a,b)\in I_1\cup I_2$.

\begin{theorem}\label{th3.6}
For $(a,b)\in \sigma$, $D_g$ satisfies the chain condition if and only if $(a,b)\in I_1\cup I_2$.
\end{theorem}

We separate the proof of Theorem \ref{th3.6} into the following two lemmas.
\begin{lemma}\label{lemma3.5}
For $(a,b)\in I_1\cup I_2$, $D_g$ satisfies the chain condition.
\end{lemma}
\begin{lemma}\label{lemma3.6}
For $(a,b)\in \sigma\setminus(I_1\cup I_2)$, $D_g$ does not satisfy the chain condition.
\end{lemma}

\begin{proof}[Proof of Lemma \ref{lemma3.5}.]
We separate the proof into two cases, that is $(a,b)\in I_1$ and $(a,b)\in I_2$. By the definition of chain condition, it suffices to show that there exists $c_0>1$ and a sequence of increasing positive integers $\{n_k\}_{k\geq 1}$ with $\lim_{k\to\infty} n_k=+\infty$ and $\frac{n_{k+1}}{n_k}\leq c_0$ for all $k$, such that \eqref{eq:chain-cond} holds for all $n_k$'s. We will prove this by inductive construction.

\medskip
\textbf{Case 1.} $(a,b)\in I_1$. Let $p\not=q\in K$. Without loss of generality, we may assume that $|x_p-x_q|\geq |y_p-y_q|$. Let $m\in \mathbb{Z}^+$ satisfying $3^{-m}< |x_p-x_q|\leq 3^{-m+1}$. Then it is easy to check that there exists a positive integer $2\leq n_1\leq 4$, and
a sequence of points $\{q_{1,i}\}_{i=0}^{n_1}\subset K$, and a chain $\gamma_1=\big(w(1,1),\ldots,w(1,n_1)\big)$ between $p$ and $q$, with $|w(1,i)|=m$ for all $i$, and $q_{1,0}=p,q_{1,n_1}=q$, both $q_{1,i-1}$ and $q_{1,i}$ are contained in $K_{w(1,i)}$ for $1\leq i\leq n_1$, and $\overline{q_{1,i-1}q_{1,i}}$ is  one of the four edges of the square $F_{w(1,i)}([0,1]^2)$ for $1< i<n_1$. In the case that $n_1=2$, we require that $q_{1,1}$ is both a vertex of $F_{w(1,1)}([0,1]^2)$ and a vertex of $F_{w(1,2)}([0,1]^2)$. Since $D_{g}$ is $1$-adapted to $g$,  there exists a constant $c_1>0$ which is only dependent on  $a$ and $b$, such that $g(\gamma_1)\leq c_1 D_g(p,q)$.

Assume that for some positive integer $k$, we have already constructed a finite sequence $\{q_{k,i}\}_{i=0}^{n_k}\subset K$ and a chain $\gamma_k=\{w(k,i)\}_{i=1}^{n_k}\subset \Sigma^*$ between $p$ and $q$, with the property that $q_{k,0}=p$, $q_{k,n_k}=q$, both $q_{k,i-1}$ and $q_{k,i}$ are two points in $K_{w(k,i)}$ for $1\leq i\leq n_k,$ and $\overline{q_{k,i-1}q_{k,i}}$ is  one of the four edges of the square $F_{w(k,i)}([0,1]^2)$ for $1< i<n_k$.

Now we will insert at most four points to $\{q_{k,i}\}_{i=0}^{n_k}$ to obtain a new sequence $\{q_{k+1,i}\}_{i=0}^{n_{k+1}}$ and a corresponding chain $\gamma_{k+1}$. Pick $\tau_k\in \{1,2,\ldots,n_k\}$, such that
\begin{equation}\label{eq:ik-pick}
   g(w(k,\tau_k))=\max\{g(w(k,j)):\, 1\leq j\leq n_k\}.
\end{equation}

In the case that $1<\tau_k<n_k$, from the inductive construction, $\overline{q_{k,\tau_{k-1}}q_{k,\tau_k}}$ is one of the four edges of the square $F_{w(k,\tau_k)}([0,1]^2)$.  Let $q_{k+1,j}=q_{k,j}$ for $j< \tau_k$,  $q_{k+1,j+2}=q_{k,j}$ for $j\geq \tau_k$, and
\begin{equation}\label{eq:insert-pts}
     q_{k+1,\tau_k}=\frac{2}{3}q_{k,\tau_k-1}+\frac{1}{3}q_{k,\tau_k}, \quad q_{k+1,\tau_k+1}=\frac{1}{3}q_{k,\tau_k-1}+\frac{2}{3}q_{k,\tau_k}.
\end{equation}
Denote $w(k+1,j)=w(k,j)$ for $j< \tau_k$ and $w(k+1,j+2)=w(k,j)$ for $j> \tau_k$. For $j=\tau_k,\tau_k+1,\tau_k+2$, we denote by $w(k+1,j)$ the subcell of $w(k,\tau_k)$ containing $q_{k+1,j-1}$ and $q_{k+1,j}$. Let $n_{k+1}=n_k+2$ and $\gamma_{k+1}=\big(w(k+1,1),\ldots,w(k+1,n_{k+1})\big)$.
From $2a+b=1$, we have $g(\gamma_{k+1})=g(\gamma_k)$.

In the case that $\tau_k=1$, from the inductive construction, we have $q_{k,0},q_{k,1}\in K_{w(k,1)}$. Then there exists an integer $1\leq \ell_k\leq 5$ and points $q_{k+1,0}, q_{k+1,1},\ldots,q_{k+1,\ell_k}$, and subcells $w(k+1,j)$, $j=1,\ldots,\ell_k$, of $w(k,1)$ with $|w(k+1,j)|=|w(k,1)|+1$ for all $1\leq j\leq \ell_k$, such that $q_{k+1,0}=q_{k,0}$, $q_{k+1,\ell_k}=q_{k,1}$, and $q_{k+1,0},q_{k+1,1}\in K_{w(k+1,1)}$, and $\overline{q_{k+1,j-1}q_{k+1,j}}$ is one of the four edges of the $F_{w(k+1,j)}([0,1]^2)$ for $j=2,\ldots,\ell_k$. In the case that $\ell_k=2$, we require that $q_{k+1,1}$ is both a vertex of $F_{w(k+1,1)}([0,1]^2)$ and a vertex of  $F_{w(k+1,2)}([0,1]^2)$. From $2a+b=1$, we can also require that
$$\sum_{j=1}^{\ell_k} g(w_{k+1,j})\leq (1+a+b)g(w_{k,1}).$$
Denote $q_{k+1,j+\ell_k-1}=q_{k,j}$ and $w_{k+1,j+\ell_k-1}=w_{k,j}$ for $j\geq 2$. Define $n_{k+1}=n_k+\ell_k-1$ and $\gamma_{k+1}=\big(w(k+1,1),\ldots,w(k+1,n_{k+1})\big)$. Then
\begin{equation}\label{eq:chain-cond-proof}
  g(\gamma_{k+1})\leq (a+b)g(w_{k,1})+g(\gamma_k).
\end{equation}

In the case that $\tau_k=n_k$, we can do similarly as in the case $\tau_k=1$.

By induction, for each positive integer $k$, we obtain a sequence of points $\{q_{k,i}\}_{i=0}^{n_k}$ and a chain $\gamma_k=\big(w(k,1),\ldots,w(k,n_k)\big)$. From $n_1\geq 2$ and $0\leq n_{k+1}-n_k\leq 4$, we have $n_{k+1}/n_k\leq 3$. From the definition of $\gamma_k$ and \eqref{eq:chain-cond-proof}, we can obtain that
\begin{equation*}
  g(\gamma_k)\leq (a+b)\max\{g(w_{1,1}), g(w_{1,n_1})\}\sum_{i=0}^\infty b^i+g(\gamma_1)\leq \Big(\frac{a+b}{1-b}+1\Big)g(\gamma_1).
\end{equation*}

Denote $c_2=c_1\big(1+(a+b)/(1-b)\big)$. Then $g(\gamma_k)\leq c_2 D_g(p,q)$ for all $k$.

Notice that $g(w(1,i))\leq (b/a)^2 g(w(1,j))$ for all $i,j\in\{1,\ldots,n_1\}$. Define
\begin{equation*}
  k_0=\min\{k\in \mathbb{Z}^+:\, b^k (b/a)^2<1\}.
\end{equation*}
Then in the case that $k\geq  k_0$, we have $w(k,j)\not\in \{w(1,1),\ldots,w(1,n)\}$ for all $1\leq j\leq n_k$. That is, for each $1\leq j\leq n_k$, there exists $j^\prime < k$, such that $w(k,j)$ is a son of $w(j^\prime,\tau_{j^\prime})$. It implies that for $1\leq i\leq n_k$,
\begin{equation*}
  g(w(k,i))\leq g(w(j^\prime,\tau_{j^\prime}))\leq a^{-1} g(w(k,j)), \quad j\in\{1,\ldots,n_k\}.
\end{equation*}
Thus
\begin{equation*}
  g(w(k,i))\leq \frac{a^{-1}}{n_k}\sum_{j=1}^{n_k} g(w(k,j)) = \frac{a^{-1}}{n_k}g(\gamma_k)\leq \frac{a^{-1}c_2}{n_k}D_g(p,q)
\end{equation*}
so that $D_g(q_{k,i-1},q_{k,i})\leq g(w(k,i))\leq  \frac{a^{-1}c_2}{n_k}D_g(p,q)$.

In the case that $k<k_0$, we have $n_k\leq 4k+2< 4k_0+2$. Thus, for each $1\leq i\leq n_k$,
\begin{equation*}
  D_g(q_{k,i-1},q_{k,i})\leq g(\gamma_k)\leq c_2 D_g(p,q)\leq \frac{(4k_0+2)c_2}{n_k} D_g(p,q)
\end{equation*}
so that the chain condition holds.

\textbf{Case 2.} $(a,b)\in I_2$. We use the same method as in Case 1. The difference is in the following.

In the first step, for each $1< i< n_1$, we require that $\overline{q_{1,i-1}q_{1,i}}$ is one of following line segments contained in the square $F_{w(1,i)}([0,1]^2)$: the endpoints of the line segment are the midpoint of edges of the square, and the slope of the line segment takes values in $\{1,-1\}$.


From step $k$ to step $k+1$, if $1<\tau_k<n_k$, then we use \eqref{eq:insert-pts} to insert two points. By using $2b+a=1$, we have $g(\gamma_{k+1})=g(\gamma_k)$. If $\tau_k=1$, we insert points similarly to the first step. That is, we require that for all $j=2,\ldots,\ell_k$,  $\overline{q_{k+1,j-1}q_{k+1,j}}$ is one of following line segments in the square $F_{w(k+1,j)}([0,1]^2)$: the endpoints of the line segment are the midpoint of edges of the square, and the slope of the line segment takes values in $\{1,-1\}$. The case $\tau_k=n_k$ is same as the case $\tau_k=1$.

Using the same argument as in Case~1, we can see that the chain condition holds.
\end{proof}

\begin{proof}[Proof of Lemma \ref{lemma3.6}.]
We prove the lemma by using the argument of contradiction. From $\sigma\setminus(I_1\cup I_2)=(\sigma_1\setminus I_1)\cup (\sigma_2\setminus I_2)$, we consider the following two cases.

\medskip
\textbf{Case 1.} $(a,b)\in \sigma_1\setminus I_1$, i.e., $b\geq a$ and $2a+b>1$. We assume that $D_g$ satisfies the chain condition.
Let $\widetilde a=a^\lambda$ and $\widetilde b=b^\lambda$, where $\lambda>1$ is the unique solution of the equation
\begin{equation*}
2a^\lambda+b^\lambda=1,
\end{equation*}
then $(\widetilde{a},\widetilde{b})\in I_1$. Let  $\widetilde{g}$ and $D_{\widetilde{g}}$ be the weight function and its induced metric associated with  $(\widetilde{a},\widetilde{b})$. Clearly we have
\begin{equation*}
\widetilde{g}(w)=g(w)^\lambda,
\end{equation*}
for any $w\in\Sigma^*$,
and hence for any chain $\gamma=\big(w(1),w(2),\ldots,w(m)\big)$,
\begin{equation*}
\widetilde{g}(\gamma)=\sum\limits_{i=1}^m\widetilde{g}\big(w(i)\big)=\sum\limits_{i=1}^m g\big(w(i)\big)^\lambda\leq \left(\sum\limits_{i=1}^m g\big(w(i)\big)\right)^\lambda=g(\gamma)^\lambda.
\end{equation*}
This implies that for any $p,q\in K$, we have
\begin{equation}\label{eqeq11}
D_{\widetilde{g}}(p,q)\leq D_{g}(p,q)^\lambda.
\end{equation}

Fix $p\neq q\in K$. Then both $D_{g}(p,q)>0$ and $D_{\widetilde{g}}(p,q)>0$ hold. By the chain condition of $D_{g}$ as we assumed, there is $C>0$ such that for any positive integer $n$, there is a chain $p=q_0,q_1,\cdots,q_n=q$ such that
\begin{equation*}
D_{g}(q_i,q_{i+1})\leq C\frac{D_{g}(p,q)}n, \ \text{ for }\ i=0,1,\cdots,n-1.
\end{equation*}
By using \eqref{eqeq11} for each pair $(q_i,q_{i+1})$, we have
\begin{equation*}
D_{\widetilde{g}}(q_i,q_{i+1})\leq D_{g}(q_i,q_{i+1})^\lambda \leq  \left(C\frac{D_{g}(p,q)}n\right)^{\lambda},
\end{equation*}
and hence
\begin{equation*}
D_{\widetilde{g}}(p,q)\leq\sum\limits_{i=0}^{n-1}D_{\widetilde{g}}(q_i,q_{i+1})\leq n\left(C\frac{D_{g}(p,q)}n\right)^{\lambda}=n^{1-\lambda}\big(C\cdot D_{g}(p,q)\big)^{\lambda}.
\end{equation*}
By letting $n\rightarrow\infty$, we obtain that $D_{\widetilde{g}}(p,q)=0$. This is a contradiction to the fact that $D_{\widetilde{g}}(p,q)>0$.

\medskip
\textbf{Case 2.} $(a,b)\in \sigma_2\setminus I_2$, i.e., $a\geq b$ and $a+2b>1$. By letting $\widetilde a=a^\lambda$ and $\widetilde b=b^\lambda$ with $\lambda>1$ satisfying
\begin{equation*}
a^\lambda+2b^\lambda=1,
\end{equation*}
and using the same argument as in Case 1, we can also find  that $D_{g}$ does not satisfy the chain condition.
\end{proof}

Let $X_t$ be the standard Brownian motion on the Sierpinski carpet $K$ constructed in \cite{BB89, KZ}. Let $\mu$ be a self-similar measure on $K$ with positive probability weights $\{\mu_i\}_{i=1}^8$ satisfying $\sum_{i=1}^8\mu_i=1$.
From \cite[Theorem 3.4.5]{Kig09}, $\mu$ is volume doubling if and only if $\mu_1=\mu_3=\mu_5=\mu_7$, $\mu_2=\mu_6$ and $\mu_4=\mu_8$.
Here we just deal with the more restricted case that $\mu_2=\mu_4$, and hence $\mu_1=\mu_3=\mu_5=\mu_7$, $\mu_2=\mu_4=\mu_6=\mu_8$, and $\mu_1+\mu_2=1/4$. We now study the time change of the standard Brownian motion on $K$ via the measure $\mu$.

Let $\rho$ be the renormalization factor of the associated Dirichlet form $(\mathcal E,\mathcal F)$ on $L^2(K,\mu)$.
Since $\rho>1$, we define the \textit{effective resistance} $R(x,y)$ between $x$ and $y$ for $x,y\in K$ as follows:
if $x=y$, define $R(x,y)=0$; if $x\neq y$, let
\begin{equation*}
R(x,y)^{-1}:=\inf\{\mathcal E(u):\ u\in\mathcal F,u(x)=0,u(y)=1\}.
\end{equation*}
Then $R(\cdot,\cdot)$ is a metric on $K$, such that for any $x,y\in K$,
\begin{equation*}
R(x,y)\asymp |x-y|^\gamma,
\end{equation*}
where $\gamma=\log\rho/\log3$.

To obtain the two-sided sub-Gaussian heat kernel estimates, we intend to find a proper metric $d$ satisfying both the chain condition, and the following property: \textit{there exists a positive real number $\beta$ such that for all $x,y\in K$,
\begin{equation*}
R(x,y)\cdot \mu\Big(B\big(x,d(x,y)\big)\Big)\asymp d(x,y)^\beta.
\end{equation*}}
A good choice is to use the metrics constructed via self-similar weight functions.

 We simplify this into the following: find $a$ and $b$ such that $(a,b)\in I_1\cup I_2$, and for any integer $n\geq0$ and for any finite word $w$ with $|w|=n$,
\begin{equation*}
\rho^{-n}\cdot \mu_w=g_{a,b}(w)^\beta.
\end{equation*}
By solving this equation we get
\begin{equation}\label{eqab}
a=(\mu_1/\rho)^\frac1{\beta},\ b=(\mu_2/\rho)^\frac1{\beta},
\end{equation}
where $\beta$ is the unique positive number satisfying:
\begin{equation}\label{eqbeta}
\left(\frac{\mu_1\vee \mu_2}{\rho}\right)^\frac1{\beta}+2\left(\frac{\mu_1\wedge\mu_2}{\rho}\right)^\frac1{\beta}=1,
\end{equation}
with $\mu_1\vee \mu_2=\max\{\ \mu_1,\mu_2\}$ and $\mu_1\wedge \mu_2=\min\{\ \mu_1,\mu_2\}$.

\medskip

Let $d=D_{g_{a,b}}$ be the metric constructed by $g_{a,b}$ with $(a,b)$ as in \eqref{eqab}. Then by applying \cite[Theorem 3.2.3]{Kig09}, the heat kernel of the Dirichlet form $(\mathcal E,\mathcal{F})$ on $L^2(K,\mu)$ exists, and enjoys the following two-sided sub-Gaussian estimate:
\begin{equation*}
p_t(x,y)\asymp \frac{1}{V(x,t^{1/\beta})}\exp\left(-c\left(\frac{d(x,y)}{t^{1/\beta}}\right)^{\frac{\beta}{\beta-1}}\right),
\end{equation*}
where $\beta$ is given by the equation \eqref{eqbeta}, and $V(x,t^{1/\beta})=\mu(\{z: d(x,z)\leq t^{1/\beta}\})$.

\medskip

\noindent{\bf Acknowledgment.} The authors are indebted to Professor Jun Kigami for his valuable suggestions and comments, especially in bringing their attention to several references about the adaptedness and quasisymmetry of the metrics.

\bibliographystyle{amsplain}

\end{document}